\documentclass[12pt]{amsart}
\usepackage{amssymb}

\usepackage{color}

\def\i#1{i^{(#1)}}
\def\j#1{j^{(#1)}}
\def\k#1{k^{(#1)}}
\def\d#1{d^{(#1)}}

\newtheorem{lem}{Lemma}[section]
\newtheorem{rem}[lem]{Remark}
\newtheorem{prop}[lem]{Proposition}
\newtheorem{lemma}[lem]{Lemma}
\newtheorem{theorem}[lem]{Theorem}
\newtheorem{cor}[lem]{Corollary}

\newtheorem{defn}[lem]{Definition}
\newtheorem{ex}[lem]{Example}

\title{ Spectral sequences of operad algebras }

\author{K. Bauer}
\address{Department of Mathematics, University of Calgary, Calgary, AB, Canada}
\email{kristine@math.ucalgary.ca}

\author{L. Scull}
\address{Department of Mathematics, Fort Lewis College,  
Durango, Colorado,  
USA
}

\email{scull\_l@fortlewis.edu}

\input xy
\xyoption{all}

\begin{document}

\newtheorem{example}{Example}[section]

\newcommand{\Lie}{{\rm Lie}}
\newcommand{\J}{{\bf J}}
\newcommand{\oh}{{\bf o}}
\newcommand{\one}{{\bf 1}}
\newcommand{\R}{{\mathbb R}}                 
\newcommand{\w}{{\bf v}}                                    
\newcommand{\ww}{{\bf w}}
\newcommand{\x}{{\bf x}}
\renewcommand{\smash}{\wedge}

\newcommand{\C}{ C_{*,*}}
\newcommand{\A}{ A_{*, *}}
\newcommand{\s}{\mathcal S}

\begin{abstract} 
We identify conditions under which it is guaranteed that an action of an operad  on the $E^2$ page of a spectral sequence passes to $E^r$ for $r\ge 2$ and hence to the  $E^\infty$ page.   We consider this question in both the purely algebraic and topological settings.  \end{abstract}

\maketitle

%--------------------------
\section{Introduction}
%--------------------------

The existence of a product structure can greatly aid in spectral sequence computations, and incorporating this structure into a spectral sequence is a standard tool.
  What makes this possible is that it is well-understood when a spectral sequence is compatible with the algebraic structure present -- when a product structure on an $E^2$ page of a spectral sequence will result in a product on the $E^\infty$ page.  
This was done, for example, by Massey \cite{Mas}, who defined  an infinite sequence of conditions which need to be satisfied, one for each page of the spectral sequence.  However, Massey's conditions are not easy to work with, as each condition needs to be verified seperately.  Other  conditions are summarized by McCleary  in \cite{McC} and presented nicely for topological spectral sequences by Dugger in  \cite{D}.  In these cases, inductive conditions are established which ensure that each differential of the spectral sequence will satisfy a Leibniz rule, ensuring that the multiplicative structure persists from one page to the next.  

In this paper, we consider algebraic operations described more generally by operads.  An operad is designed to encode a set of operations and their composition laws.  This gives a framework for considering many different kinds of algebraic structures from a common viewpoint, as actions of various operads.   More general constructions, such as PROPs, are also considered in section 3.   Operads and spectral sequences have both become ubiquitous in algebra and topology.   We present conditions under which an action of an operad
  on the $E^2$ page of a spectral sequence passes to the  the $E^\infty$ page.  These conditions generalize the conditions for algebras, and open the door for new applications.  It seems to be a folk theorem that spectral sequences exist for operad algebras (see, e.g. Section 1.4 of \cite{AK}).  However, no explicit conditions for such appear in the literature.  Our goal is to fill this gap.

 In this paper, we explore three circumstances in which a spectral sequence can arise from operad algebras.  The first situation is purely algebraic.  In Section \ref{s:algebra} we consider an exact couple arising from a short exact sequence of chain complexes, each of which is an algebra over a differential graded operad.  The other two situations are topological.  In Section \ref{s:spaces}, we consider exact couples arising from a tower of spaces.  In this case, the entire tower is an algebra over the operad.  Finally, in Section \ref{s:spectra} we explore the stable setting.  
 
The authors would like to acknowledge the debt owed to the work of Dan Dugger \cite{D} on product structures in spectral sequences.  We have used the methods in that paper, as applied to operad algebras, throughout our work.  We would also like to thank the referee for extensive suggestions and improvements.

%--------------------------
\section{Preliminaries}
%--------------------------

In this section, we briefly  review the definitions of operads and operad algebras, and set up the context in which we will be working.    Two good references for this material are \cite{F} and \cite{M}.

For concreteness, we will consider operads and algebras in one of two underlying categories {\sf C}.  In algebraic settings,   ${\sf C}={\sf dg-Mod}_k$, the category  of differential graded $k$-modules (shortened to dg modules), where $k$ is a commutative  ground ring.  For topological considerations, we use   ${\sf C}= {\sf Tops}$, the category of basepointed topological spaces.  Many of the results discussed here can be extended without difficulty to an arbitrary symmetric monoidal category.  

We will denote a dg module by $(A,d)$ where $d$ is the differential of the module $A$, or just by $A$.  The tensor product $\otimes$ will denote $\otimes_k$, unless otherwise specified.  The degree of an element $u$ of a differential graded module is  denoted $|u|$.  

\subsection{Operads}    An operad is designed to encode a system of algebraic  operations and their compositions.  In our category {\sf C}, we have an object (which is either a dg module or a space) of  $n$-ary operations for each $n$.   One way to encode this information is via a symmetric sequence.  In the category ${\sf dg-Mod}_k$ this is   a sequence $\{(P(n), \delta_n)\}_{n\geq 1}$ of dg-modules $P(n)$ with differential $\delta_n$ such that each $P(n)$ is equipped with an action of the $n$-th symmetric group $\Sigma_n$ (which we can think of as changing the order on the input variables).   For simplicity, we will denote the differential of $P(n)$ by $\delta$ when the context makes it clear which $n$ is involved.  In $\sf Tops$ a symmetric sequence would simply be a sequence of spaces $X_n$ equipped  with an action of $\Sigma_n$.    An operad in {\sf C} is a symmetric sequence which is a monoid with respect to the product:  \[(P\circ P)(n)= \bigoplus_{k\geq 1} P(k)\otimes \Big( \bigoplus P(j_1)\otimes \cdots \otimes P(j_k) \Big) \]
where the second direct sum is taken over all ways of writing \\ $n=j_1 + \cdots + j_k$.
That is,  $P$ is an operad if there exists a map \[\gamma: P\circ P\to P\] which is unital, associative and equivariant (see for example \cite{F}, \cite{M}).  Since we think of this map as  the composition of operations,  it is convenient to denote the image of \[(p, q_1, \ldots , q_k) \in P(k)\otimes P(j_1)\otimes \cdots \otimes P(j_k)\] by $p(q_1, \ldots , q_k)\in P(j_1+\cdots + j_k)=P(n)$.  

 Note that when ${\sf C}={\sf dg-Mod}_k$, this is a map of dg modules.  Therefore,  it must also satisfy the derivation relation
\[ \delta(p(q_1, \ldots, q_k))=\delta(p)(q_1, \ldots , q_k) + \sum_{i=1}^k \pm (q_1, \ldots, \delta(q_i), \ldots , q_k)\]
where  the sign in the sum is given by $(-1)^{|p|+|q_1|+\cdots +|q_{i-1}|}$ for each $i$.  This property is also called the Leibniz rule.  On first reading, the reader may want to simplify to the case   where each $P(n)$ is concentrated in degree 0, so that $P$ is just an operad of $k$-modules.  

\subsection{Operad algebras} 
An algebra over the operad $P$ is an object which has the operations described by $P$.  Precisely, it is an object $A$ of ${\sf C}$ equipped with maps $\Gamma:P(k)\otimes A^{\otimes k}\to A$ satisfying the usual associativity and equivariance axioms (see e.g. Definition \ref{d:toweraction}, or \cite{M}).  We denote the image of $(p, x_1, \ldots, x_k)$ under this map by $p(x_1, \ldots , x_k)$.  When ${\sf C}={\sf dg-Mod}_k$, $p(x_1, \ldots , x_k)$ has degree $|x_1|+\cdots +|x_k|$.    Again, the structure maps are maps of dg modules and so satisfy a derivation relation:
\[ d(p(x_1, \ldots , x_k))= \delta(p)(x_1, \ldots, x_k) + \sum_{i=1}^k \pm p(x_1,\ldots , d(x_i) ,\ldots , x_k)\]\label{e:derivation}%
with  sign determined as before: for the $i$-th term, it is   $(-1)^{|p|+|x_1|+\cdots +|x_{i-1}|}$.
In addition, the operad maps are in particular morphisms of $k$-modules, and so we have the following additive properties:    \[ p(x_1, \ldots, x_i+y, \ldots  , x_n)=p(x_1, \ldots, x_i, \ldots , x_n)+p(x_1,\ldots ,y, \ldots, x_n)\]  and \[(p+q)(x_1, \ldots , x_n) = p(x_1, \ldots , x_n) + q(x_1, \ldots , x_n).\]
 
Under these assumptions, operad actions are compatible with the differential structure, as the next result shows.     

\begin{lemma} \label{l:action1} Let $(A, d)$ be an algebra over the operad $(P, \delta)$.  Then the homology $H(A, d)$ is an algebra over the homology operad $H(P, \delta)$.\end{lemma}

\begin{proof}  Let $[p]$ be a homology class of  $H(P, \delta)$,  with representative  $p$  in $P(k)$.  Let $[x_i]$ be homology classes of $H(A,d)$ for $1\leq i \leq k$.  Then we define $[p]([x_1], \ldots , [x_k])=[p(x_1, \ldots , x_k)]$ where $x_i$ is a representative in $A$ for each $[x_i]$.  The additive properties of the action mentioned above can easily be used to show that  varying the representatives of  $p$ or the $x_i$'s by a boundary results in the same homology class  $[p(x_1 , \ldots , x_k)]$.

\end{proof}

\subsection{Examples} 
The graded commutative associative operad $Comm(n)$ is created from equivalence classes of connected planar binary trees with one root and $n$ labelled leaves;   two such trees are equivalent if they are related by a sequence of adjacent moves on the leaves (corresponding to switching brackets using associativity) or by switching the order of adjacent leaves with appropriate sign (corresponding to graded commutativity).   An algebra over this operad is exactly a graded commutative associative algebra  

Many other algebraic structures can be encoded by similar operads;   for example, associative algebras are encoded by the operation of the operad $Assoc$ created in an analogous way to $Comm$ (but with a suitibly altered equivalence relation);  various other algebraic structures have their corresponding operads,  such as  Lie algebras encoded by the operad $Lie$.  Many examples of operads, their algebras and their uses can be found in \cite{MSS}.

%---------------------------------
\section{Spectral sequences in algebra} \label{s:algebra}
%---------------------------------
In this section we consider spectral sequences of dg modules which are also algebras over a dg operad $P$.  We give conditions which ensure that we produce a spectral sequence of operad algebras.
%we will explain which ensure that we  get a spectral sequence of operad algebras. 
We examine operad algebras whose underlying structure is that of a bigraded module $A_{p,q}$ with differential $d_A:A_{p,q}\to A_{p, q-1}$.  In addition to satisfying the derivation relation from the previous section, therefore, we will also require the operad action to respect both gradings: if $x_i\in A_{p_i, q_i}$, $1\leq i\leq k$, then $p(x_1, \ldots , x_k)$ has degree $(p_1+\cdots + p_k, q_1+\cdots +q_k)$.

 Our approach to spectral sequences will be through exact couples which arise from certain short exact sequences of chain complexes.  An exact couple is a 5-tuple  $\langle E, D, i, j, k\rangle$:  \begin{equation}
 \xymatrix { D \ar[rr]^i && D \ar[dl]^j \\ & E \ar[ul]^k}
 \end{equation}\label{e: couple}%
 where $E$ and $D$ are $k$-modules (possibly graded or bigraded) with $k$-module maps between them,  and the diagram is exact at each corner.   If in addition $E$ is a dg operad algebra under the differential $jk$, we will call this an exact couple of operad algebras; then Lemma \ref{l:action1} will ensure that the derived couple inherits an operad action on $E'$ as well.   This does not automatically mean that the  derived couple will again be an exact couple of operad algebras;  we need additional conditions so that  the derived map $j'k'$ will again satisfy the Leibniz condition to be a map of dg operad algebras.   %  I think the definition of spectral sequence of operad algebras should be included in the introduction when we explain the pont of the paper (even if we haven't defined an operad algebra yet).  

One typical way in which an exact couple arises is from short exact sequences of chain complexes, and this is the setting we will be considering.    A short exact sequence of the form 
\begin{equation}
 \xymatrix{0 \ar[r] & A \ar[r]^i & A \ar[r]^j & C \ar[r] & 0 }\label{e:ses}
 \end{equation}
gives rise to a long exact sequence on homology which in turn becomes an exact couple with $D=H(A)$ and $E=H(C)$,  and the map $k$ given by the connecting homomorphism.       There are various choices of grading for this;  our conventions will be that the 
bigraded modules $A_{*,*}$ and $C_{*, *}$ have vertical differential $d_A$ and $d_C$ (that is, these are both maps of degree $(0,-1)$), the injection $i$ has degree $(1,0)$, and the surjection $j$ has degree $(0, 0)$.   This choice of grading leads to a homology spectral sequence, one can just as easily choose grading which produce a cohomological version.

We can introduce an operad action into this picture as follows.  Suppose that $(P, \delta)$ is an operad of of bigraded modules with one vertical differential and no horizontal differential, and that  $\A$ and $\C$ are operad algebras over $P$.  Moreover, suppose that the chain complex maps $i$ and $j$ in the short exact sequence are compatible with the operad action, in the following sense:  
\begin{itemize}
\item[(i)] The surjection $j$  satisfies $j(p(x_1, \ldots , x_k))=p(j(x_1), \ldots , j(x_k))$.
\item[(ii)] The injection $i$ satisfies $i(p(x_1, \ldots , x_k))=p(x_1, \ldots, i(x_h), \ldots , x_k)$ for any $0\leq h \leq k$
\end{itemize}

At first, the condition on the map $i$ may seem strange; it becomes more natural when we remember that the spectral sequence associated to this situation is graded by images of $i$. For example, if $A$ comes with a filtration $F_0 \subseteq F_1 \subseteq F_2 \dots \subseteq A $  and $i$ is an inclusion map then this condition is satisfied:  raising the filtration degree  of one of the inputs of the operation will raise the filtration degree of the result.   Another example where this condition is naturally satisfied comes from the Bockstein spectral sequence:  for $q$ a prime,  $A = \oplus_n Q$  is a coproduct of a $q$-torsion free dg operad algebra $Q$, and $i:  A \to A$ is the  ``multiply by  $q$'' map of algebras.  Again, we think of filtering by powers of $q$, and  $i$ adds one more power of $q$ regardless of where it is applied.

%any examples from pairings of towers?  How about Whitehead's filtered spaces?

The main goal of this section is  the following,  
\begin{theorem} \label{t:main}   If a spectral sequence arises from a short exact sequence of chain  complexes $0 \to A \overset{i}{\to} A \overset{j}{\to} C \to 0$ which are $P$-algebras for an operad $P$, such that $i$ and $j$ satisfy properties (i) and (ii) as above, then the resulting spectral sequence inherits the operad algebra structure at all levels $E^r$.  

\end{theorem}

Before presenting the proof of Theorem \ref{t:main}, we first recall a few details about the construction of exact couples.

\begin{rem} \label{r:cyclesboundaries}
{\em To pass from the short exact sequence of equation (\ref{e:ses})  to the (first) exact couple, we begin with  \[\langle E^{1}, D^{1} \rangle  =\langle H(C, d_C), H(A, d_A), i^{(1)}, j^{(1)}, k^{(1)}\rangle,\] where $i^{(1)}$, $j^{(1)}$ are the maps on homology induced by the maps $i$ and $j$ and $k^{(1)}$ is the connecting homomorphism.  This connecting homomorphism is defined by taking a pre-image of an element under $j$, applying $d$, and taking a pre-image under $i$;  abusing notation, we will write this as  $k = i^{-1}dj^{-1}$.  
This exact couple then gives rise to a spectral sequence  by successively defining the derived couples: $\langle E^2, D^2, \i{2}, \j{2}, \k{2} \rangle$ is given by  $E^2=H(E^1, jk)$ and $D^2=i(D)$, with    $\i{2}$ and $\k{2}$ induced by $i$ and $k$ with little change, while $\j{2}(a)=[j(i^{-1}(a))]$ (where again $i^{-1}(a)$ denotes a pre-image of $a$ under $i$).    Iterating this process results in a further derived couple for each $r>1$ and in general, only the derived maps $\j{r}$ change significantly: $\j{r}(a)=[j(i^{-(r-1)}(a))]$.   The kernel and image of $\d{r} = j^{(r)} k^{(r)} $ can be identified at each stage in terms of the initial maps $i$, $j$ and $k$, and we obtain:
\[ 0 \subseteq B^2 \subseteq 
\cdots \subseteq B^r \subseteq \cdots \subseteq Z^r \subseteq \cdots \subseteq Z^2 \subseteq  E^1 \]
where $Z^r=k^{-1}(i^{r-1}D)$ is the kernel of $j^{(r)}k^{(r)}$,  $B^r=j({\rm ker} (i^{r-1}))$ is the image, and  $E^r=Z^r/B^r$.}  
\end{rem}

We will  show that each derivation $\d{r} = \j{r}\k{r}$ satisfies the Leibniz condition.  An easy induction with Lemma \ref{l:action1} will show that each $\langle E^r, D^r \rangle $ is an exact couple of operad algebras.

\begin{proof}[Proof of Theorem \ref{t:main}] Let $p\in H(P(\ell))$ and let $[x_h]$, $1\leq h \leq \ell$, be elements of $E^{r}$.  Then each $[x_h]$ has a representative $x_h$ in $C$, since  $Z^r\subset E^1=H(C)$.  Choose $a_h\in A$ with $j(a_h)=x_h$ for each $h$.      We compute $k^{(r)}[p(x_1, \ldots , x_\ell)]$ by first computing $k(p(x_1, \ldots , x_\ell))$:
 \begin{flalign}  
k (p(x_1, \ldots , x_{\ell}) )
& = i^{-1}d j^{-1} (p(x_1, \ldots , x_{\ell})) \\ 
 & =  i^{-1} d (p(a_1, \dots, a_{\ell})) 
  \end{flalign}
using property (i) defining $j$.
Now $  H(P, \delta)$  has no higher differential, and we treat it as a chain complex with $0$ differential.  Thus, $d(p)=0$ and we have  
 \begin{flalign}
 %\phantom{k (p(x_1, \ldots , x_{\ell}) )}
i^{-1}d(p(a_1, \ldots , a_l)]
& =   \sum_{h=1}^\ell \pm i^{-1} p(a_1, \ldots, d(a_h) , \ldots , a_{\ell}) \\ 
& =  \sum_{h=1}^\ell \pm p(a_1, \ldots, i^{-1} d(a_h) , \ldots , a_{\ell})\\
   & = \sum_{h=1}^\ell\pm p(a_1, \ldots , k(x_h), \ldots, a_{\ell}) 
   %\\ & =  i^{r-1} \sum_{h=1}^\ell \pm p(a_1, \ldots , a_h', \ldots, a_{\ell}) 
   \end{flalign} 
   where the sign for the $i$-th summand is given by $(-1)^{|p|+|x_1|+\cdots +|x_{i-1}|}$.  Here we have used property (ii) defining the map $i$, and the definition of $k(x_h)$.  
   %Now, we specified at the outset that $x_h\in Z^r$, so by definition of $Z^r$, we can find $a_h'\in A$ such that $i^{r} (a_h') = d(a_h)$; so that $k(x_h)= i^{r-1}(a_h')$.   
   Using our computation of $k^{(r)}[p(x_1, \ldots , x_l)]$, we compute  $ j^{(r)}k^{(r)}[p(x_1, \ldots , x_h)]$ by remembering that $j^{(r)}[a]= [j(i^{-(r-1)}a)]$.  Thus $j\left(i^{-(r-1)}\sum_{h=1}^\ell\pm p(a_1, \ldots , k(x_h), \ldots, a_{\ell})\right) $ is equal to:
   \begin{flalign}
 \sum_{h=1}^\ell \pm j(p(a_1, \ldots, i^{-(r-1)} k(x_h), \ldots, a_{\ell}) \quad \rm{by\ property \ (ii)} \\
 = \sum_{h=1}^\ell \pm p(j(a_1), \ldots, j(i^{-(r-1)}k(x_h)), \ldots, j(a_{\ell}) \quad \rm {by\ property \ (i)} \end{flalign}
with the signs given as before.  Upon passing to homology, this shows that 
$\j{r}\k{r} [p(x_1, \ldots , x_{\ell})]$ is equal to 
\[ \sum_{h=1}^\ell \pm p([x_1], \ldots , \j{r}\k{r} [x_h] , \ldots , [x_{\ell}] )
\]
as desired.

%Let $[x]$, $[y]$ be elements of $E^{r+1}$.  Thus $[x]$ and $[y]$ have representatives $x,y\in Z^{r+1}$.  Since $x,y\in E$ we have $x=j(a)$ and $y=j(b)$.  Let $k(x)=i^r(a')$ and $k(y)=i^r(b')$ ($a', b'$ exist by definition of $Z^{r+1}$.  By the definition of the derived couples, $\k{r}[xy]=k(xy)$ which we computed in the previous Lemma to be $i^{r-1}(a')b \pm ai^{r-1}(b')$.  Then $\j{r}(k(xy))=\j{r}(i^{r-1}(a'b)\pm i^{r-1}(ab')$ (using the property we placed on $i$) which is in turn equal to $[j(a'b)\pm j(ab')] = [j(a')][j(b)]\pm [j(a)][j(b')]=\j{r}\k{r}[x][y]\pm [x]\j{r}\k{r}[y]$, which is what we wanted.

This gives us exactly what we need to finish the proof that each derived exact couple $ \langle E^r, D^r, \i{r}, \j{r}, \k{r} \rangle$ in the spectral sequence is an exact couple of operad algebras.   We induct on $r$:  since the first derived couple $\langle E^2, D^2, \i{2}, \j{2}, \k{2} \rangle$ is defined by  $E^2=H(E^1, d)$ and $D^2=i(D)$, we see that $E^2$ inherits the operad algebra structure by Lemma \ref{l:action1} and $D^2$ by the fact that $i$ respects the operad action.  Moreover, we have shown that $d^{(2)} $ satisfies the Leibniz condition, the first derived couple is also an exact couple of operad actions.

Similarly, each exact couple is obtained from the previous by iterating this process.  So if we assume that  $ \langle E^r, D^r, \i{r}, \j{r}, \k{r} \rangle$ is an exact couple of operad actions,  Lemma \ref{l:action1} and the compatibility of $i$ ensures that the derived exact couple $ \langle E^{r+1}, D^{r+1}, \i{r+1}, \j{r+1}, \k{r+1} \rangle$ consists of operad algebras, and our previous check shows that the derived differential $j^{(r+1)}k^{(r+1)}$ continues to behave well and satisfy the necessary Leibniz condition.

Therefore every term in the spectral sequence $E^r$ will inherit the operad action.
\end{proof}

\subsection{Examples}

Specializing to commutative algebras via the $Comm$ operad, Theorem \ref{t:main}  exactly states that under the conditions we have just described the product structure on the $E^r$ term descends to the $E^{r+1}$-term, a standard result about algebraic structures.    Using another operad such as $Assoc$ or $Lie$ would give analogous results for the corresponding associative or Lie algebra structures.  

By taking the operad  approach, Theorem \ref{t:main} has identified conditions for a general   algebraic structure present  to be compatible with the spectral sequence, and we  do not have to consider each algebraic context seperately.

\subsection{Convergence}

As stated in the wonderful introduction to spectral sequences (\cite{McC}, page 5), a spectral sequence is used precisely to relate something computable to something desirable.  In this section, we aim to show that if we have a spectral sequence of operad algebras, then the ``something desirable" will also be an operad algebra.
 In the case of spectral sequences arising from short exact sequences of chain complexes $0\to A\to A\to C\to 0$ which we have been considering, the spectral sequence converges to a limit or colimit of the graded complexes $A_{p,q}$.  We will describe the situation for colimits; both cases are well studied in the literature and the interested reader might pursue \cite{W}, \cite{D} for further information.

In Remark \ref{r:cyclesboundaries}, we used the cycles $Z^r$ and boundaries $B^r$ to define $E^r = Z^r/B^r$.   A spectral sequence is thus determined by the nested sequence
\[ 0 \subseteq B^2 \subseteq 
\cdots \subseteq B^r \subseteq \cdots \subseteq Z^r \subseteq \cdots \subseteq Z^2 \subseteq  E^1. \]
This sequence determines a unique object
\[ E^\infty = \left(\underset{r}{\cap} Z^r\right) / \left( \underset{r}{\cup} B^r \right).\]
We say that the spectral sequence {\it converges} to some graded object $H_*$ if there is an isomorphism $E^\infty_{p,q} \cong H_{q}$.  (Here we are working in the category of modules, so it is guaranteed that ${\cap_r}Z^r$ and ${\cup_r} B^r$ exist.  In general, one may need to impose extra hypotheses.  See \cite{W}, p. 125.)    In this section, we show that $E^\infty$ is an operad algebra, identify a good candidate for $H_*$ for our spectral sequence, show that $H_*$ is also an operad algebra and finally show that the operad algebra structures on $E^\infty$ and $H_*$ must agree.

\begin{rem} {\em Most references using a different grading, so that the spectral sequence converges to $H_*$ if there is an isomorphism $E^\infty_{p,q}\cong H_{p+q}$.  In the spectral sequence studied in Theorem \ref{t:main}, one can arrange for this by using the grading with  $D_{p,q}=A_{p, p+q}$ instead of $D_{p,q} = A_{p,q}$.}
\end{rem}

Note that Thereom \ref{t:main} implies that $E^\infty$ must be an operad algebra.  Suppose that we have elements $[{x_1}], \ldots , [{x_\ell}]$ of $E^\infty$ and an element ${p}$ in $H(P(n), \delta)$.
We can arrange to choose representatives $x_i$ of $[{x_i}]$ each in some $Z^r$ (with $r$ fixed).  Then
\[ {p}([{x_1}],\ldots , [{x_\ell]}) = [{{p}(x_1, \ldots , x_l)}] \]
is a well-defined action of $H(P, \delta)$ on $E^\infty$.  To see that this is well-defined, just note that any two representatives $x_i$ and $x_i'$ of $[{x_i}]$ must differ by something in $\underset{r}{\cup} B^r$.  However, Theorem \ref{t:main} shows  that the collection of maps $d^{(r)}= j^{(r)}k^{(r)}$ satisfy the Leibniz rule, which is precisely the condition needed to show that the choice of representative for $[{x_i}]$ does not matter.

Now we'll turn to the problem of determining what the ``something useful" should be.  In any exact couple
\[ \xymatrix{ D \ar[rr] && D \ar[dl] \\ & E \ar[ul] }\]
the ``something computable" is related to whatever $E$ might be.  It should not  be surprising, then, that the ``something desirable" is related to $D$.  We will present the desirable quantity for our particular setting of spectral sequences arising from $0\to A \overset{i}{\to} A \overset{j}{\to} C\to 0$, but much of what is  here can easily be translated into other settings.  To discover what the spectral sequence might converge to, we filter $D$ by images of the map $i: A\to A$.  Let
\[ H_q := \underset{p}{colim}\ A_{p, q}\]
where the colimit is taken over the maps $i:A_{p, q}\to A_{p+1, q}$, etc.  The graded objects $H_q$ are filtered;  we let $F_pH_q$ be the image of $A_{p,q}$ in $H_q$.  The $F_pH_q$'s form an exhaustive filtration of $H_q$ with 
\[ F_{p-1}H_q \subset F_pH_q\subset \cdots \subset H_q.\]
We consider the associated graded objects $F_pH_q/F_{p-1}H_q$.

\begin{prop}[\cite{W}, Proposition 5.9.6] \label{p:weibel}There is a natural inclusion  $\gamma : F_pH_q/F_{p-1}H_q\to E^\infty_{p, q}$. \end{prop}

The proof of the proposition is essentially due to the fact that for any $a\in A_{p,q}$ which represents an element of $F_pH_q/F_{p-1}H_q$, $j(a)$ is an infinite cycle.  This induces the map $\gamma$.  The details of the proof are outlined in \cite{D} and proved in detail in \cite{W}.
In good circumstances, the spectral sequence will converge to $H_*$.  What we mean by ``good circumstances" depends on the spectral sequence, but we'll give one example of a good convergence criterion here.  The exact couple is {\it bounded below} if for each $q$ there
exists a $p(q)$ such that $A_{p, q}=0$ whenever $p<p(q)$.  

\begin{theorem}[\cite{W}, Theorem 5.9.7]  If an exact couple is bounded below, then the spectral sequence converges to $H_*$. \end{theorem}

Regardless of whether or not the spectral sequence converges to $H_*$, the associated graded modules are an algebra over the homology operad $H(P, \delta)$.  

\begin{theorem}  If $\{ E^r_{p,q}\}$ is a spectral sequence satisfying the hypotheses of Theorem \ref{t:main}, then the collection of associated graded modules $\{ F_pH_q/F_{p-1}H_q \}$ is an algebra over the operad $H(P, \delta)$.\end{theorem}

\begin{proof} Let $p\in H(P(\ell), \delta$ and  $y_i\in F_{p_i}H_{q_i}/F_{p_i-1}H_{q_i}$ for $1\leq i\leq \ell$.  We can find representatives $a_i\in A_{p_i, q_i}$ of $y_i$.  First define $p(y_1, \ldots , y_\ell)$ to be the image of $p(a_1, \ldots , a_\ell)$ in $H_q$.  Note that $p(a_1, \ldots , a_\ell)\in A_{p, q}$ where $p=p_1+ \cdots + p_\ell$ and $q=q_1+\cdots + q_\ell$, so that the image is contained in $F_pH_q$.  If $a_i$ and $a_i'$ are two different representatives of $y_i$, then their difference is in $F_{p_i-1}H_{q_i}$.  It is straightforward to check that $p(a_1, \ldots , a_i , \ldots, a_l)$ and $p(a_1, \ldots , a_i', \ldots , a_n)$ vary by elements of $F_{p-1}H_q$, making this definition of the product well-defined.
\end{proof}

In light of Proposition \ref{p:weibel}, we can investigate the relationship between the operad action on $E^\infty$ and the operad action on $F_pH/F_{p-1}H$.  

\begin{cor} \label{t:convergence} The map $\gamma$ is a map of operad algebras.
\end{cor}

\begin{proof} The proof follows immediately, since $\gamma$ is induced by the map $j$ which is a map of operad algebras by condition (ii) from the definition of $j$.
\end{proof}

When the spectral sequence converges to $H_*$, Theorem \ref{t:convergence} guarantees that the spectral sequence converges {\it as operad algebras}.  

%%%%%%%%%%%%%%%%%%%%%%%%%%%%%%%%%%%%%%%%%%%%%%%%%%%%%%%%%%%%%%%%%%%%%%%%%%%%%%%%%%%%%%%%%%%%%%%%%%%%%

\subsection{PROPS}

Operads are a special case of a more general construction, that of a PROP ({\bf pro}duct and {\bf p}ermutation category)  encoding operations with multiple inputs and multiple outputs.      A PROP is a
collection $P =\{P(m, n)\}$ for $  m, n \geq 0$  of  $(\Sigma_m, \Sigma_n)$-bimodules, where $\Sigma_m$ acts on the left and $\Sigma_n$ ono the right and the two actions commute.  We think of $P(m,n)$ as representing a space of maps which have $m$ inputs and $n$ outputs, where $\Sigma_n$ switches the order of the inputs and $\Sigma_n$ the outputs.   We require two  types of compositions,
horizontal
$$\otimes:    P(m_1, n_1) \otimes \cdots \otimes P(m_s, n_s) \to P(m_1 + \cdots + m_s, n_1 + \cdots+ n_s)$$
 and vertical
$$\circ : P(m, n) \otimes P(n, k) \to  P(m, k)$$
and a unit element ${\bf 1} \in P(1,1)$ satisfying certain conditions (see for example \cite{MSS} for details).

Similarly, a  $P$-algebra $A$ is a collection of maps $$ \{ \Gamma:  P(m, n) \otimes A^{\otimes m} \to A^{\otimes n} \}_{m, n \geq 0} $$
that are compatible with the horizontal and the vertical compositions, the unit elements, and the $\Sigma_n-\Sigma_m$ actions.

As before, if we have a PROP $P$ of  $ \sf {dg-Mod}_k$, we can consider spectral sequences of $ \sf {dg-Mod}_k$ algebras over $P$ and write down  the conditions needed for the PROP structure to persist throughout the spectral sequence.    The conditions are exactly the same as those listed above for operad algebras.  Since we assume that the structure maps $\Gamma$ are maps of $ \sf {dg-Mod}_k$-modules, they commute with the differential defined via a Leibniz condition on   $P(m, n) \otimes A^{\otimes m}$ and $A^{\otimes n}$.  If we assume that we have a short exact sequence of PROP algebras leading to an exact couple and that the injection $i$ and surjection $j$ satisfy the conditions (i) and (ii) given before Theorem \ref{t:main}, the proof still holds if we just interpret $p(x_1, \dots, x_n)$ as an element in $A^{\otimes n}$.  Therefore our proof also shows how to get a spectral sequence of PROP algebras.

\begin{example}  The endomorphism PROP of a  $k$-module $V$ is the system
$EndV = \{EndV (m, n)\}_{m,n\geq 0} $ 
with $EndV (m, n)$ the space of linear maps $ Lin(V^{\otimes n}, V^{\otimes m}) $ with n ÔinputsÕ and m Ôoutputs,Õ
${\bf{1}}  \in  EndV (1, 1)$ the identity map, horizontal composition given by the tensor product of linear
maps, and vertical composition by the ordinary composition of linear maps.  \end{example}
%----------------------------------------------------------------------------
\section{Spectral sequences in topology} \label{s:spaces}
%----------------------------------------------------------------------------

   Now we will consider spectral sequences of basepointed topological spaces.  We choose our category of topological spaces to be nice enough so that the smash product is associative;  for example, the category of compactly generated spaces will do (this is the category used in \cite{EKMM} and \cite{HSS}, for example).  A basepointed subspace $A$ of a space $B$ gives rise to a long exact sequence
\begin{equation}\label{e:les} \xymatrix{ \cdots \ar[r] & \pi_n(A) \ar[r] & \pi_n(B) \ar[r] & \pi_n(B, A) \ar[r] & \cdots}\end{equation}
where $\pi_*(B, A)$ are the relative homotopy groups.  Elements in $\pi_n(B,A)$ are homotopy classes of commuting diagrams
\[ \xymatrix{ S^{n-1} \ar[r]  \ar@{^{(}->}[d] & A \ar@{^{(}->}[d] \\ D^n \ar[r] & B}.\]  A sequence of inclusions
\[ \dots \subset A_n \subset A_{n-1} \subset \cdots \subset A_1 \subset A_0 = B\]
gives rise to long exact sequences for each $A_n \subset A_{n-1}$; these fit together to produce an exact couple with $D_{p,q}= \pi_q(A_p)$ and $E_{p,q}=\pi_q(A_{p-1}, A_p)$.  In the next section, we will show how this exact couple arises explicitly.

\subsection{Towers of spaces.}  In topology,  spectral sequences often arise more generally from a tower of pointed spaces and maps which are not necessarily inclusions.   In this subsection, we will review the basic constructions needed to produce a spectral sequences from a tower of pointed spaces.  An excellent reference for this material is \cite{D}, or see any number of introductions to spectral sequences such as \cite{McC}.  

We consider a tower of pointed spaces and  maps $W$:

\[ \xymatrix{ \cdots \ar[r] & W_{m+1} \ar[r]^{q_{m}}& W_m \ar[r]^{q_{m-1}}  & W_{m-1} \ar[r] & \cdots\ar[r]^{q_0} & W_0}\]

%In practice, such a tower is either {\it bounded below} ($W_k=\star$ for all values of $k$ which are small enough) or {\it bounded abover} ($W_k=W_{k+1}$ and the map between these is the identity for all $k$ large enough).   
For any spaces $A$ and $B$ with a map $f:A\to B$, we can generalize the relative homotopy groups by defining $\pi_n(B, A)$ to be the set of homotopy classes of commuting diagrams
\[ \xymatrix{ S^{n-1} \ar[r]  \ar@{^{(}->}[d] & A \ar[d]^{f} \\ D^n \ar[r] & B}.\]
These relative homotopy groups are isomorphic to the homotopy groups of the homotopy fiber of the map $f$, which is defined to be the homotopy pull-back of the diagram
\[ \xymatrix{  & \star \ar[d]  \\ A \ar[r]^{f} & B. } \]
Thus, these are groups for $n\geq 1$, abelian groups for $n\geq 2$ and fit into a long exact sequence of homotopy groups similar to the one in equation (\ref{e:les}).    In particular, it will be useful later to consider $\pi_*(W_m, W_{m+t})$ for each $t\geq 1$ and the map $W_{m+t} \to W_m$ defined by $q^t=q_m \circ \cdots \circ q_{m+t-1}$.  We denote the tower $W$ with grading shifted by $t$ as $W_{+t}$.
For each $t$, assembling the long exact sequences for $\pi_*(W_m, W_{m+t})$ for each $m\geq 0$ results in an exact couple of homotopy groups for any tower $W$:
\begin{equation} \label{e:extcpl}\xymatrix{ \pi_*(W) \ar[rr]^i && \pi_*(W)\ar[dl]^j\\
& \pi_*(W, W_{+t})\ar[ul]^k }\end{equation}
with $i=\pi_*(q^t)$.  When $t=1$, this often gives rise to a spectral sequence, although there may be problems with exactness since  $\pi_0$ might not be a group and $\pi_1$ might not be an abelian group.  We will assume that we are in a situation where these problems do not cause any concern (e.g. all spaces are connected enough).   The $E^1$ term of the spectral sequence we are interested in will correspond to the case $t=1$; however it will be useful to define the maps $j$ and $k$ for the more general case $t\geq 1$.

To define the maps $j$ and $k$ in the exact couple more precisely, begin by selecting orientation preserving homeomorphisms $S^n\cong D^n/S^{n-1}$ for all $n$. An element $[\alpha] \in \pi_n(W_k, W_{k+t})$ is represented by a diagram
\[ \xymatrix{ S^{n-1} \ar@{^{(}->}[d]\ar[r]^{\alpha_0} & W_{k+t} \ar[d]^{q^t} \\ D^n \ar[r]^{\alpha_1} & W_k}.\]  
Then $k[\alpha]\in \pi_{n-1}(W_{k+t})$ is represented by $\alpha_0$.  It is convenient to think of this as an element of the relative homotopy group $\pi_{n-1}(W_{k+t}, \star)$ represented by 
\[ \xymatrix{ S^{n-2} \ar[r] \ar@{^{(}->}[d] & \star \ar[d] \\ D^{n-1} \ar[r]^{\alpha_0'} & W_{k+t} }\]  
where $\alpha_0'$ is the map $\alpha_0$ precomposed with the quotient map $D^{n-1}\to D^{n-1}/S^{n-2}$ and our preselected homeomorphism $D^{n-1}/S^{n-2} \cong S^{n-1}$.  This is the boundary map of the long exact sequence for the pair $(W_{k}, W_{k+t})$.  Similarly, if $[\beta]$ is an element of $\pi_n(W_k)$  represented by
\[ \xymatrix{ S^{n-1} \ar[r] \ar@{^{(}->}[d] & \star \ar[d] \\ D^n \ar[r]^{\beta_1} & W_k} \] in $\pi_*(W_k, \star)$ then $j[\beta]$ is represented by
\[ \xymatrix{ S^{n-1} \ar[r]^{\star} \ar@{^{(}->}[d] & W_{k+t} \ar[d]^{q^t} \\ D^n \ar[r]^{\beta_1} & W_k} \]
in $\pi_n(W_k, W_{k+t})$ where $\star$ is the unique map which factors through the one-point space $\star$.  With these definitions,  one can see that $i$, $j$ and $k$ satisfy the exactness conditions needed to ensure that the triangle in (\ref{e:extcpl}) is an exact couple.

The differential of the exact couple is given as usual by the composition $jk$.  To produce the $E^2$ page of the associated spectral sequence, we take $E^2= \operatorname{ker} jk/ \operatorname{im} jk$.  
In particular, taking $t=1$ we have $E^1_{p,q} = \pi_p(W_{q-1}, W_{q})$.
Subsequent pages of the spectral sequence are constructed in a similar manner using the derived maps $j^{(r)}$ and $k^{(r)}$.  The differentials $d^{(r)}$ are $d^{(r)}=j^{(r)}k^{(r)}$ and satisfy  \[j^{(r)}k^{(r)}: E^r_{p,q}\to E^r_{p-1, q+r}.\]
It will be useful to have a way to recognize elements of $\pi_*(W, W_{+1})$ which survive to the $E^r$ page of the associated spectral sequence.  As in Remark 3.2, define subgroups
\[ Z^r_{p,q} = \operatorname{ker} \left(j^{(r)}k^{(r)}: E_r^{p,q} \to E_r^{p-1, q+r}\right).\]
In order to identify those elements of $\pi_*(W, W_{+1})$ which survive to $E^r$, a good first step is to understand the elements which lie in each $Z^r$.  The following proposition appears as Lemma 3.3 of  \cite{D}.

\begin{lem}\label{l:cycles}  An element $\alpha\in \pi_p(W, W_{+1})$ lies in $Z^r_{p,q}$ if and only if $\alpha$ can be represented as a diagram $D$:
\[ \xymatrix{ S^{p-1} \ar@{^{(}->}[d] \ar[r]^{\alpha_0} & W_{k+1} \ar[d]^{q_k} \\ 
D^p \ar[r]^{\alpha_1} & W_k } \]  
in which the map $\alpha_0:S^{p-1} \to W_{k+1}$ factors through $W_{k+r}$.
\begin{proof}
We provide here only an idea of the proof for $r=2$, as the full details can be found in \cite{D}.   If $\alpha$ lies in $Z^2_{p,q}$, then $jk(\alpha)$ must be null.  If $\alpha$ is represented by a diagram $D$ as in the statement of the Lemma,  then $jk(\alpha)$ is the diagram
\[ \xymatrix{ S^{p-2} \ar[r]^\star \ar@{^{(}->}[d] & W_{k+2} \ar[d] ^q \\ D^{p-1} \ar[r]^{\alpha_0'} & W_{k+1}}.\]
This diagram represents the zero element in homotopy if and only if the induced quotient map
\[ S^{p-1} = D^{p-1}/S^{p-2} \to W_{k+1}/W_{k+2} \]
is null-homotopic.  But since this map is induced by $\alpha_0$, this happens exactly when $\alpha_0$ factors through $W_{k+2}$.
However the factorization need only be up to homotopy, and
the rest of the proof  proceeds by applying the homotopy extension property to rigidify this argument.
\end{proof}
\end{lem}

\begin{rem} If $\alpha\in Z^r$ can be represented as a square diagram with a lift
\[ \xymatrix{ & W_{k+r}\ar[d]^{q^{r-1}} \\ S^{p-1}\ar[ur] \ar@{^{(}->}[d] \ar[r]^{\alpha_0} & W_{k+1} \ar[d]^{q_k} \\ 
D^p \ar[r]^{\alpha_1} & W_k } \]  
then $d^{(r)}([\alpha])$ is $jk$ applied to the ``outer" square diagram
\[ \xymatrix{ S^{p-1} \ar@{^{(}->}[d] \ar[r]^{\alpha_0} & W_{k+r} \ar[d]^{q^r} \\ 
D^p \ar[r]^{\alpha_1} & W_k } .\]  
\end{rem}

\subsection{Operads acting on towers of spaces.}  Now suppose that $P = \{ P(i) \}_{i\geq 1}$ is an operad of basepointed spaces.  

\begin{defn}\label{d:toweraction}  We say that $P$ acts on the tower $W$ if, for each $i\geq 1$ and each partition $i= i_1 + \ldots + i_k$, there are maps
\[\gamma: P(k)\wedge W_{i_1} \wedge \cdots \wedge W_{i_k} \to W_i \]
which strictly commute with the structure maps $q$ of the tower $W$, and which satisfy
\begin{enumerate}
\item Equivariance:  The symmetric group $\Sigma_k$ acts on $P(k)$ since $P$ is an operad, and acts on $W_{i_1} \wedge \cdots \wedge W_{i_k}$ by permuting the factors.  The map $\gamma$ is $\Sigma_k$-fixed with respect to the action of $\Sigma_k$ on 
\[ P(k)\wedge W_{i_1} \wedge \cdots \wedge W_{i_k}\]
in which $\sigma\in\Sigma_k $ acts by $\sigma\wedge \sigma^{-1}$.
\item Associativity: The two evident composition maps
\[ P(k) \wedge (P(m_1)\wedge\cdots \wedge P(m_k)) \wedge (\bigwedge_{t=1}^m W_{i_t}) \to W_i,\]
where $m=m_1 +\cdots + m_k$ and $i=i_1 +\cdots + i_m$, are equal.
\end{enumerate}For further details, see \cite{MSS}.
\end{defn}

\subsection{Products in relative homotopy}
In this section, we show that given a tower $W$ which admits an action by an operad $P$, the spectral sequence arising from the exact couple of equation (\ref{e:extcpl}) is a spectral sequence of $P$-algebras.   The first step is to explain in which sense $\pi_*(W, W_{+t})$ is a $P$-algebra.   To accomplish this, we first have to understand products of classes in the relative homotopy groups $\pi_*(W, W_{+t})$.
%
%The bulk of what follows is devoted to showing that  the conditions for this operad algebra structure are satisfied.   In particular, showing the associativity condition will require some carefully chosen decompositions of spheres and  disks as colimits of other similar spaces. 
%
 %Once we have the operad structure in place,  we go on to show that the differential $jk$ satisfies the appropriate Leibniz rule, guaranteeing that the homology $H_*(\pi_*(W, W_+), jk)$ is again a $P$-algebra.    This result will follow from work of Dan Dugger \cite{D} on defining differentials on homotopy groups of spheres using pinch and fold maps.    This procedure is easily iterated, guaranteeing that each derived couple inherits an action of $P$ from the previous exact couple.
%
This product structure on relative homotopy groups is explained in \cite{D} and we recall the basic construction.  

Suppose that $f:A\to B$ and $g:C\to D$ are basepointed maps of topological spaces, and let $x\in \pi_n(B, A)$ and $y\in \pi_m(D, C)$ be represented by diagrams
\[ \xymatrix{ S^{n-1} \ar[r]^{x_0} \ar[d] & A \ar[d]^f  && S^{m-1}\ar[r]^{y_0}\ar[d] & C\ar[d]^{g} \\ 
D^n\ar[r]^{x_1} & B && D^m\ar[r]^{y_1} & D}\]
respectively.  We will  define their product $x\cdot y$ in $\pi_{n+m}(B\wedge D, M_{f\wedge g})$ where $M_{f\wedge g}$ is the pushout
\begin{equation}\label{eq:P(f,g)} \xymatrix{ A\wedge C\ar[r] \ar[d] & B\wedge C \ar[d] \\  A\wedge D \ar[r] & M_{f\wedge g}}.\end{equation}
%We use $P_{f\wedge g}$ instead of $A\wedge D \coprod_{A\wedge C} B\wedge C$ for brevity.
 The pushout $M_{f\wedge g}$ is equipped with a cannonical map $M_{f\wedge g}\to B\wedge D$, given by the universal property of the pushout diagram. The product  $x\cdot y$ is defined by
\[ \xymatrix{ S^{n+m-1} \ar[rr]^{\cong\qquad\qquad} \ar[d]&& S^{n-1}\wedge D^m \underset{\small{S^{n-1}\wedge S^{m-1}}}{\coprod } D^n\wedge S^{m-1}\ar[d] \ar[rr]^{\qquad \qquad \qquad x_0\wedge y_1 + x_1\wedge y_0} && M_{f\wedge g}\ar[d] \\ 
D^{n+m} \ar[rr]^{\cong} && D^n\wedge D^m \ar[rr]^{x_1\wedge y_1} && B\wedge D}.\]  
We denote the top map by $(x\cdot y)_0$ and the bottom map by $(x\cdot y)_1$.  Note that $x\cdot y = y\cdot x$ since the smash product $\wedge$ is symmetric.

We would like to show that this product is associative, but first we have to make sense out of what this means.  Suppose that $h:E\to F$ is another basepointed map of spaces, and that $z\in \pi_l(F, E)$ is represented by 
\[ \xymatrix{ S^{l-1} \ar[r]^{z_0} \ar[d] & E\ar[d]^{h}\\ D^l \ar[r]^{z_1} & F}.\]
Now consider the two elements $(x\cdot y)\cdot z$ and $x\cdot (y\cdot z)$.  The first is represented by a diagram
{\footnotesize \[ \xymatrix{  {S^{n+m+l-1}}\ar[r]^{\cong\qquad \qquad \qquad}\ar[d] 
&  {S^{n+m-1}\wedge D^l}\underset{ {S^{n+m-1}\wedge S^{l-1}}}{\coprod}  {D^{n+m}\wedge S^{l-1}} \ar[rr]^{\qquad \qquad \qquad \qquad {(x\cdot y)_0\wedge z_1 + (x\cdot y)_1\wedge z_0}} \ar[d] 
&&  {\qquad M_{(f\wedge g)\wedge h}}\ar[d] \\ 
 {D^{n+m+l}}\ar[r]^{\cong} &  {D^{n+m}\wedge D^l} \ar[rr] &&  {(B\wedge D)\wedge F}}\] }
where $M_{(f\wedge g)\wedge h}$ is the pushout
\begin{equation}\label{eq:P(f,g),h} \xymatrix{ M_{f\wedge g} \wedge E \ar[r] \ar[d] & (B\wedge D)\wedge E \ar[d]^{\alpha} \\ 
M_{f\wedge g} \wedge F \ar[r]^{\beta} & M_{(f\wedge g)\wedge h} }.\end{equation}
On the other hand, $x\cdot (y\cdot z)$ is represented by
{\footnotesize 
\[ \xymatrix{  {S^{n+m+l-1}}\ar[r]^{\cong\qquad \qquad \qquad} \ar[d] 
&  {S^{n-1}\wedge D^{m+l} \underset{S^{n-1}\wedge S^{m+l-1}}{\coprod} D^{n}\wedge S^{m+l-1}} \ar[rr]^{\qquad \qquad \qquad \qquad {x_0\wedge (y\cdot z)_1 + x_1\cdot (y\cdot z)_0}} \ar[d] && 
 {\qquad M_{f\wedge (g\wedge h)}}\ar[d] \\ 
 {D^{n+m+l}}\ar[r]^{\cong} &  {D^{n}\wedge D^{m+l}} \ar[rr] &&  {B\wedge (D\wedge F)}}\] }
where $M_{f\wedge (g\wedge h)}$ is the pushout
\[ \xymatrix{ A\wedge M_{g\wedge h}  \ar[r] \ar[d] & B\wedge M_{g\wedge h} \ar[d] \\ A \wedge (D\wedge F) \ar[r] & M_{f\wedge (g\wedge h)} }.\]

Our goal is to show that this product is associative, which asserts that the groups \[\pi_{n+m+l}(B\wedge D\wedge F, M_{(f\wedge g)\wedge h})\quad \text{and}\quad  \pi_{n+m+l}(B\wedge D\wedge F, M_{f\wedge(g\wedge h)})\] are isomorphic, and that the isomorphism carries $(x\cdot y)\cdot z$ to $x\cdot (y\cdot z)$.  To show this, we will be making use of some careful decompositions of spheres and  disks and their boundaries as colimits.

To show that $(x\cdot y)\cdot z$ is the same as $x\cdot (y\cdot z)$ it will be convenient to instead work with a three-fold product $x\cdot y\cdot z$.  To define the iterated product $(x\cdot y)\cdot z$, we made use of the fact that the $n+m-1$ sphere can be obtained from the $n-1$ and $m-1$ spheres by the pushout diagram
\[ \xymatrix{ S^{n-1}\wedge S^{m-1} \ar[r] \ar[d] & D^n \wedge S^{m-1} \ar[d] \\ S^{n-1} \wedge D^m \ar[r] & S^{n+m-1}}.\]
This is a familiar construction, which amounts to saying that the boundary of the $n+m$ ball $D^{n+m}$ can be written as a union of $\partial D^{n}\wedge D^m$ and $D^{n} \wedge \partial D^{m}$ provided that we keep track of the fact that these should be glued together along $\partial D^n \wedge \partial D^m$.  This construction can be generalized to any number of factors.  
Suppose that  $D^n$ is a product of smaller disks via a homeomorphism
\[ D^n\cong D^{n_0}\wedge D^{n_1} \wedge \cdots \wedge D^{n_k}.\]
Then $S^{n-1}$ is homeomorphic to a quotient of 
\[ \coprod_{0\leq i\leq k} D^{n_0} \wedge \cdots \wedge S^{n_i} \wedge\cdots \wedge D^{n_k}\] 
whose quotient data is obtained by understanding how to attach the boundary pieces in this sum to one another.  
This can be done by taking the colimit of a cubical diagram.  
\begin{defn} \label{d:cubical} Let $[k]$ be the set $\{0, \ldots, k\}$ and let ${\mathbb P}_0([k])$ be the power set of {\it proper} subsets of $[k]$.  Then define a functor $\chi:{\mathbb P}_0([k]) \to Top_*$ by 
\[ \chi(U)=X_0(U)\wedge \cdots X_k(U)\] 
for each $U\subsetneq [k]$, where
\[ X_i(U) = \begin{cases} D^{n_i} & \text{if $i\in U$ and} \\ S^{n_i -1} & \text{if $i\notin U$.} \end{cases}\] 
\end{defn}

The functor $\chi$ takes an inclusion of subsets to the map of topological spaces obtained by inclusion of the boundary.  Each such map carries the assembly data we require, and the boundary of $D^n$ is homeomorphic to the colimit
\[ \underset{U\subset {\mathbb P}([k])}{\operatorname{colim}} \chi .\]

\begin{ex} The boundary of the cube $D^1\times D^1\times D^1$ is homeomorphic to the 2-sphere, $S^2$.  Definition \ref{d:cubical} simply encodes the information needed to express the 2-sphere as the union of its parallel faces glued along common edges.  The edges, in turn, are glued along common vertices.  All of this information is encoded in the cube $\chi$.
\end{ex}

\begin{lem} \label{l:assoc} The product $\cdot$ defined above is associative.
\end{lem}

\begin{proof}  Let $f:A\to B$, $g:C\to D$ and $h:E\to F$ and choose elements $x$, $y$ and $z$ as in the preceding paragraphs.  
To prove the assertion, we show that both $M_{(f\wedge g)\wedge h}$ and $M_{f\wedge (g\wedge h)}$ are homeomorphic to a third space, $M$.  Define $M$ to be the colimit of the punctured cubical diagram
{\footnotesize \[ \xymatrix{ A\wedge C \wedge E \ar[rr] \ar[dr] \ar[dd] && A\wedge C\wedge F \ar[dr] \ar'[d][dd]\\
& A\wedge D\wedge E \ar[dd] \ar[rr] && A\wedge D\wedge F\\
B\wedge C\wedge E \ar'[r][rr] \ar[dr] && B\wedge C\wedge F \\
&B\wedge D \wedge E}\]}

\noindent where each set of three parallel maps is given by $f$, $g$ or $h$.  Note that there is a unique map from $M$ to $B\wedge D\wedge F$.

To show that $M\cong M_{(f\wedge g)\wedge h}$, we use the Fubinic Theorem for colimits \cite{MacLane} which states that for small categories $I$ and $J$, and for a functor ${\sc F}$ on $I\times J$, the iterated colimit and the double colimits are isomorphic:
\[ {\operatorname{colim}_{I\times J}} {\sc F} \cong {\operatorname{colim}_J}{\operatorname{colim}_I} \tilde{\sc F}\]
where $\tilde{\sc F}$ is the functor which for each $i\in I$ produces a functor $\tilde{\sc F}(i)$ on $J$, defined  by $\tilde {\sc F}(i)(j) = {\sc F}(i,j)$.

The cubical diagram defining $M$ can be written as the colimit of three diagrams, each with the shape of a pushout, as follows:

\[ \def\objectstyle{\scriptstyle} \xymatrix{ B\wedge C \wedge F & A\wedge C \wedge F \ar[l] \ar[r] & A\wedge D\wedge F  \\
B\wedge C \wedge E\ar@{-->}[u] \ar@{-->}[d] & A\wedge C \wedge E \ar[l] \ar[r] \ar@{-->}[u]\ar@{-->}[d] & A\wedge D\wedge E \ar@{-->}[u]\ar@{-->}[d]\\
B\wedge D \wedge E & B\wedge D \wedge E \ar@{=}[l] \ar@{=}[r]  & B\wedge D\wedge E 
}.\]

The pushout of the top row is $M_{f\wedge g}\wedge F$; the pushout of the middle row is $M_{f\wedge g}\wedge E$ and  the pushout of the bottom row is trivially $B\wedge D\wedge E$.  Taking the iterated pushout, we obtain $M_{(f\wedge g)\wedge h}$.  
The Fubinic Theorem says precisely that in this situation,  $M\cong M_{(f\wedge g)\wedge h}$.

The cube defining $M$ is symmetric with respect to the maps\\  $f:A\to B$, $g:C\to D$ and $h:E\to F$. Taking advantage of that symmetry we see that $M\cong M_{f\wedge (g\wedge h)}$ as well.

There exists a unique map $P\to B\wedge D\wedge F$ given by applying $f$, $g$ and $h$ to each of the terminal vertices of the cube which defines $M$.  

The sphere $S^{n+m+l-1}$ can be obtained as the boundary of\\  $D^n\wedge D^m\wedge D^l$ as in Definition \ref{d:cubical}.  The cubical diagram which describes the attaching information can be rewritten as
\begin{equation}\label{d:sphere} \def\objectstyle{\scriptstyle} \xymatrix{ S^{n-1}\wedge D^m \wedge D^l & S^{n-1}\wedge S^{m-1}\wedge D^l \ar[l] \ar[r] & D^n\wedge S^{m-1}\wedge D^l  \\
S^{n-1}\wedge D^m \wedge S^{l-1}\ar[u] \ar[d] & S^{n-1}\wedge S^{m-1}\wedge S^{l-1}\ar[l] \ar[r] \ar[u]\ar[d] & D^n\wedge S^{m-1}\wedge S^{l-1} \ar[u]\ar[d]\\
D^n\wedge D^m \wedge S^{l-1} & D^n\wedge D^m \wedge S^{l-1} \ar@{=}[l] \ar@{=}[r]  & D^n\wedge D^m\wedge S^{l-1} 
}\end{equation}
where $\chi({\emptyset}) = S^{n-1}\wedge S^{m-1}\wedge S^{l-1}$, etc.
Each of the maps in this diagram are induced by the inclusion of a boundary sphere into a disk.  The colimit of the diagram is the sphere $S^{n+m+l-1}$. 

The maps 
\begin{align*} x_0\wedge y_1\wedge z_1&: S^{n-1}\wedge D^m \wedge D^l \to M\\
x_1\wedge y_0\wedge z_1 &:D^n\wedge S^{m-1}\wedge D^l \to M\\
x_1\wedge y_1\wedge z_0 &: D^n\wedge D^m \wedge S^{l-1}\to M \end{align*}
define a map from the cubical diagram (\ref{d:sphere}) to $M$.  Since the sphere is the colimit of diagram (\ref{d:sphere}), this produces a map \[{\sum x_i\wedge y_j\wedge z_k}:S^{m+n+l-1}\to M.\]  Putting this together with the map $M\to B\wedge D\wedge F$  we produce the element $x\cdot y\cdot z\in \pi_{m+n+l-1}(M; B\wedge D\wedge F)$ represented by
\[ \xymatrix{ S^{n+m+l-1}\ar[r]^{\cong\quad} \ar[d]& \underset{U\subsetneq \{0, 1, 2\}}{\operatorname{colim}} \chi(U)\ar[r]^{\qquad \sum x_i\wedge y_j\wedge z_k} \ar[d] & P \ar[d] \\
 D^{n+m+l}\ar[r] & D^n\wedge D^m\wedge D^l\  \ar[r]^{\quad x_1\wedge y_1\wedge z_1} & \ B\wedge D\wedge F}\]
which commutes precisely because the diagrams representing the elements $x$, $y$ and $z$ commute.

By applying the Fubinic Theorem to diagram (\ref{d:sphere}), we see that the homeomorphism $M\cong M_{(f\wedge g)\wedge h}$ produces an isomorphism 
\[\pi_*(P, B\wedge D\wedge F)\cong \pi_*(P_{(f\wedge g)\wedge h}, B\wedge D\wedge F)\]
 which carries $x\cdot y\cdot z$ to $(x\cdot y)\cdot z$.  A symmetric argument shows that $x\cdot y \cdot z$ is carried to $x\cdot (y\cdot z)$ under the isomorphism induced by 
$M\cong M_{f\wedge(g\wedge h)}$.

\end{proof}

\begin{prop}\label{p:E_1algebra}  If the operad $P$ acts on the tower $W$, then $\pi_*(W, W_{+t})$ is an algebra over the operad $\pi_*(P)$ of graded groups.  That is, there are maps
\[ \pi_*(P(k))\otimes \pi_*(W, W_{+t})^{\otimes k} \to \pi_*(W, W_{+t})\]
satisfying the properties described in Section 2.2.\end{prop}

\begin{proof}  First note that the graded groups $\pi_*(P(k))$ form an operad since $\pi_*$ is a product preserving functor.

To define the maps $\pi_*(P(k))\otimes \pi_*(W, W_{+t} )^{\otimes k}\to \pi_*(W, W_{+t})$, suppose that we are given elements
$\rho\in \pi_{n_0}(P(k))$ represented by
\[ \xymatrix{ S^{n_0-1} \ar[r] \ar[d] & \star\ar[d] \\ D^{n_0} \ar[r]^{\rho_1} & P(k) }\]
and elements $\alpha_i\in \pi_{n_i}(W_{m_i}, W_{m_i+t})$, $1\leq i\leq k$, represented by
\[ \xymatrix{ S^{n_i-1} \ar[r]^{\alpha_{i_0}} \ar[d] & W_{m_i+t} \ar[d]^{q^t} \\ D^{n_i} \ar[r]^{\alpha_{i_1}} & W_{m_i} .} \]
Let $n=n_0 + n_1 + \cdots + n_k$ and $m=m_1 + \cdots + m_k$.  Our goal is to produce horizontal maps for the diagram
\[ \xymatrix{ S^{n-1} \ar[r]^{\beta_0} \ar[d] & W_{m+t} \ar[d]^{q^r} \\ D^n \ar[r]^{\beta_1} & W_m}.\]
To do this, first we write $D^n$ as a product of smaller disks via a choice of homeomorphism
\[ D^n\cong D^{n_0}\wedge D^{n_1} \wedge \cdots \wedge D^{n_k}.\]
Now we write the boundary $S^{n-1}$ of $D^n$ as the colimit of the $k$-dimensional punctured cubical diagram $\chi$ of Definition \ref{d:cubical}.

For each $U\subsetneq [k]$, define a map $\chi(U) \to P\wedge W \wedge \cdots \wedge W$ by 
\[ \xymatrix{ X_0(U) \wedge X_1(U) \wedge \cdots \wedge X_k(U)\ar[rrr]^{\qquad f_0(U)\wedge f_1(U) \wedge \cdots \wedge f_k(U)} &&& P\wedge W \wedge \cdots \wedge W}\]
where 
\[ f_i(S) = \begin{cases} \alpha_{i_1} & \text{when $i\neq 0$ and $i\in U$,}\\
                                         \alpha_{i_0} & \text{when $i\neq 0$ and $i \notin U$,}\\
                                         \rho_1 & \text{ when $i=0$ and $i\notin U$, and }\\
                                         \star & \text{ when $i=0$ and $i\in U$.} \end{cases} \]
 
Now we compose with the operad structure maps to produce maps 
\[ \chi(U) \to W\]
for each $U$.  In the case $0\in U$, this map is null.  If $0\notin U$ and if $|U|=u$, then the image of this map lies in $W_{m+(k-u+1)t}$.  These maps commute in the sense that if $U\subset U'$ with $|U|=u$ and $|U'|=u'$, then
\begin{equation} \label{d:chi(S)} \xymatrix{ \chi(U) \ar[r] \ar[d]& W_{m+(k-u+1)t} \ar[d]^q\\
 \chi(U')\ar[r] & W_{m+(k-u'+1)t} }\end{equation}
commutes, where $q$ denotes the appropriate composition of maps $q$ from the tower $W$ and the other vertical map is induced by $\chi$ applied to the inclusion $U\subset U'$.
In particular, when $S$ is maximal (i.e. $|U|=k$), then the operad action produces maps $\chi(U) \to W_{m+r}$.  The commuting of diagram (\ref{d:chi(S)}) ensures that each vertex $\chi(U)$ maps to $W_{m+t}$ by a map which factors through $W_{m+(k-u+1)t}$.  
  Thus the universal property of the colimit ensures that we have produced a map
\[ \beta_0:S^{n-1} \to W_{m+t},\]
as desired.

The map $\beta_1:D^n \to W_m$ is the composition
\[ \xymatrix{ D^{n_0} \wedge D^{n_1} \wedge \cdots \wedge D^{n_k} \ar[rr]^{\rho_1\wedge \cdots \wedge \alpha_{k_1}} && P(k) \wedge W_{m_1} \wedge \cdots \wedge W_{m_k} \ar[r] & W_m}\]
whose second map is given by the action of $P$ on $W$.

The resulting square diagram with the maps $\beta_0$ and $\beta_1$ commutes because the operad structure maps are required to commute with the maps $q$ of the tower.  The action of $\pi_*(P)$ on $\pi_*(W, W_{+t})$ is equivariant and associative because the action of $P$ on $W$ is so.  Thus we conclude that $\pi_*(W, W_{+t})$ is an algebra over $\pi_*(P)$.

\end{proof}

From now on, we denote the product of $\rho, \alpha_1, \ldots, \alpha_k$ defined in the previous proposition by $\rho(\alpha_1, \ldots, \alpha_k)$.

\subsection{Differential, derivations and the main result.}
We are now finally ready to show that the spectral sequence arising from the exact couple of equation (\ref{e:extcpl}) is a spectral sequence of operad algebras.
In order to do so, we must ensure that the differential $jk$ behaves well with regard to the structure maps coming from the action of $P$ on $W$.  That is, we need to know that $jk$ is a differential, just as in Section 2.1:
\[ jk( \rho(\alpha_1, \ldots, \alpha_k) = \sum_{i=1}^k \pm \rho(\alpha_1, \ldots, jk(\alpha_i), \ldots, \alpha_k)\]
where $\rho\in \pi_{n_0}(P(k))$ and $\alpha_i\in \pi_{n_i}(W_{m_i}, W_{m_i+t})$ for $1\leq i\leq k$.
Note that since $P$ as a pair is $(P, \star)$, the differential $jk(\rho)$ must be zero, hence the expected first term of the sum has been left out.  The sign, as before, is $(-1)^{n_0+ \cdots + n_k}$ on the $i$-th term of the summand.

\begin{prop} \label{p:jk-der} The homomorphisms $jk:\pi_*(W, W_{+t}) \to \pi_*(W, W_{+t})$ are derivations with respect to the action of $\pi_*(P)$ on $\pi_*(W, W_{+t})$.
\end{prop}

\begin{proof}  This is Proposition 4.1 of \cite{D}, together with Lemma \ref{l:assoc}, which proves that the product $\cdot$ is associative.  That is, we have
\[ jk(\rho \cdot \alpha_1\cdots \alpha_k) = jk(\rho)\cdot (\alpha_1\cdots \alpha_k) + (\pm 1)(\rho)\cdot jk(\alpha_1\cdots \alpha_k)\]
by \cite{D}.  Now continue inductively to evaluate $jk(\alpha_1\cdots \alpha_k)$.  So doing produces the desired result.  The fact that $\cdot$ is associative says that this formulation of the product in terms of iterated products is equivalent to any other formulation of the product, which completes the proof.
\end{proof}

\begin{rem}{\em  Proposition 4.1 is the heart of the product structure for spectral sequences in \cite{D}.  The proof there goes roughly as follows.  First, reduce to the case where $f:A\to B$ is $S^{p-1}\hookrightarrow D^p$ and $g:C\to D$ is $S^{q-1}\hookrightarrow D^q$.  In that case, $M$ is the pushout of \[S^{p-1}\wedge D^q \leftarrow S^{p-1}\wedge S^{q-1} \rightarrow D^p\wedge S^{q-1},\]
which is a $(p+q-1)$-sphere.  The map $M \to A\wedge C$ is essentially inclusion of the equator.    The smash product $A\wedge C$ is a $S^{p+q-2}$ sphere, and we wish to compute the three homotopy elements $j_*k(x\cdot y)$, $(kx)\cdot y$ and $x\cdot(ky)$ in $\pi_{p+1-1}(S^{p+q-1}, S^{p+q-2})$.  Define a map $D:\pi_{p+q-1}(M, A\wedge C) \to \pi_{p+1-1}(M/A\wedge C)$ by sending an element 
\[ \xymatrix{ S^{p+q-2}\ar[r]^{\alpha_0}\ar[d] & A\wedge C \ar[d]\\
D^{p+q-1}\ar[r]^{\alpha_1} & M }\]
to the quotient map $\alpha:D^{p+q-1}/S^{p+q-2} \to M/A\wedge C$.  This is a map from $S^{p+q-1} $ to $S^{p+q-1}\vee S^{p+q-1}$, hence $D:\pi_{p+q-1}(M, A\wedge C) \to {\mathbb Z}\bigoplus {\mathbb Z}$.  The proof proceeds by showing that $D$ is an injection, and that $D(jk(x\cdot y)) = (1,1)$, $D((kx)\cdot y) = (1,0)$ and $D(x\cdot(ky)) = (0, (-1)^p)$.  The images of $D$ are verified by examining the geometry of the various spheres involved - the signs, for example, arise by keeping track of the orientations of each copy of $S^{p+q-1} $ in $S^{p+q-1}\vee S^{p+q-1}$. } 
\end{rem}

\begin{theorem}  Let $\{ E^r \}$ be the spectral sequence  which arises from the exact couple of diagram (\ref{e:extcpl}) in the case $t=1$.  The operad action of $P$ on the tower $W$ induces an action of $\pi_*(P)$ on each
term $E^r$ of the spectral sequence, and the differentials $d^{(r)}$ satisfy the Leibniz rule 
\[ d^{(r)}(\rho(\alpha_1\cdots \alpha_k)) = %d^{(r)}(\rho)\cdot (\alpha_1\cdots \alpha_k)+ 
(\pm 1)(\rho)d^{(r)}(\alpha_1\cdots \alpha_k)\]
for all $\rho\in \pi_m(P)$ and $\alpha_i\in E^r_{p_i, q_i}$, $1\leq i \leq k$.  The sign is given by $(-1)^p$ where $p = p_1 + \cdots + p_k$.  
% If the spectral sequence converges, it converges as a spectral sequence of $\pi_*(P)$-algebras.
\begin{proof}  We have already proved this theorem for $r=1$ in Propositions \ref{p:E_1algebra} and \ref{p:jk-der}.  To complete the proof, note that $\alpha_i\in E^r_{p_i, q_i}$ are in particular elements of $Z^r_{p_i, q_i}$.  By Lemma \ref{l:cycles}, each of these can be represented by a square diagram with a lift:
\[ \xymatrix{ & W_{q_i+r} \ar[d] \\ S^{p_i-1} \ar[r] \ar@{^{(}->}[d] \ar[ur] & W_{q_i+1} \ar[d] \\ D^{p_i} \ar[r] & W_{q_i} }.\]
Applying Propositions \ref{p:E_1algebra} and \ref{p:jk-der} to the outer ``square" diagrams produces the desired result.

%The statement about convergence can be made precise by replacing Theorem \ref{t:main} in subsection 3.2 by the first part of Theorem 4.10.
\end{proof}
\end{theorem}

\begin{rem} \label{r:convergence} As in the algebraic case, one could consider whether or not the spectral sequence studied in this section converges as an operad algebra.  The results of subsection 3.2 can be translated into this setting with little change.  We define
\[ H_n = \underset{p}{colim} \ \pi_n(W_p) \]
and filter $H_n$ by $F_pH_n$, which is the image of $\pi_n(W_p)$ in $H_n$.  The conclusion of Theorem 3.6 holds in this context, and as the proof is not significantly different, we leave the study of convergence as an exercise for the curious reader.
\end{rem}
   %%%%%%%%%%%%%%%%%%%%%%%%%%%%%%%%%%%%%%%%%%%%%%%%%%%%%%%%%%%%%%%%%%%%%%%%%%%%%%%%%%%%%%%%%%%%%%%%%%%%%%%%%%%%%%%%%%%%%%%%%%%%%%%%%%%

\section{Spectral sequences in the stable setting}\label{s:spectra}

The results of Section \ref{s:spaces} can be extended to the stable setting by considering spectral sequences of spectra with operad actions.  Suppose that we have a tower
\[ \cdots \to W_{m+1} \to W_m \to W_{m-1} \to \cdots \]
of spectra.  In order to define an action of an operad $\{P_n\}$ of spectra on the tower $W$, we need a well-behaved product.  This allows us to form
\[ P_k\wedge W_{i_1} \wedge \cdots \wedge W_{i_k}\]
so that we may generalize Definition \ref{d:toweraction} to spectra.  At a minimum, ``well-behaved" means that this product must be associative.  There are several equivalent categories of spectra whose construction includes a symmetric monoidal product $\wedge$,  that is, a product which is associative, commutative and unital up to isomorphism.  These categories include the ring spectra of \cite{EKMM} as well as symmetric spectra \cite{HSS}.  In this exposition, we will use the latter.  

Briefly, a symmetric spectrum ${\bf X}$ is a sequence of pointed spaces $\{ X_n\}_{n\geq 0}$ together with structure maps $\sigma: S^1\smash X_n \to X_{n+1}$ and a basepoint preserving action of the $n$-th symmetric group $\Sigma_n$ on $X_n$ which keeps track of iterated structure maps.  In particular, if $\Sigma_p$ acts on $S^p = (S^1)^{\wedge p}$ by permuting factors, the map
\[ \sigma^p = \sigma\circ(S^1\wedge \sigma) \circ \cdots \circ (S^{p-1}\wedge \sigma):S^p\wedge X_n \to X_{n+p}\]
is required to be $\Sigma_p\times \Sigma_n$-equivariant for $p\geq 1$ and $n\geq 0$.  The sphere spectrum ${\bf S}= \{ S^n\}_{n\geq 0}$ is an example of a symmetric spectrum.

The smash product of two symmetric spectra ${\bf X}$ and ${\bf Y}$ is the tensor product of ${\bf X}$ and ${\bf Y}$ as modules over the sphere spectrum, ${\bf S}$.  For any symmetric spectra ${\bf X}$ and ${\bf Y}$, let 
${\bf X}\otimes {\bf Y}$ be the symmetric sequence with 
\[ ( X\otimes  Y)_n = \bigvee (\Sigma_n)_+ \wedge_{\Sigma_p\times \Sigma_q} (X_p\wedge Y_q)\]
where the sum is taken over all $p+q=n$.
If ${\bf X}$ is any symmetric spectrum then there is a map 
\[ m:{\bf S}\otimes {\bf X} \to {\bf X}\]
given by the collection of $\Sigma_p\times \Sigma_q$-equivariant maps
\[ \sigma^p: S^p \wedge X_q \to X_{p+q}.\]
Equivalently, this produces a map ${\bf X}\otimes {\bf S} \to {\bf X}$.
Let ${\bf X}\otimes {\bf S} \otimes {\bf Y}$ be the symmetric sequence with
\[ (X\otimes S \otimes Y)_n = \bigvee (\Sigma_n)_+ \wedge_{\Sigma_p \times \Sigma_q \times \Sigma_r} (X_p\wedge S^q \wedge Y_r)\]
where the sum is taken over all $p+q+r=n$.   Then the smash product of ${\bf X}$ and ${\bf Y}$ is the colimit of
\[ \xymatrix{{\bf X}\otimes {\bf S} \otimes {\bf Y} \ar@<1ex>[r]^{m\otimes 1} \ar@<-1ex>[r]_{1\otimes m} & {\bf X}\otimes {\bf Y}}.\]
The smash product is a symmetric monoidal product of symmetric spectra whose unit is ${\bf S}$.  See \cite{HSS} for further details about symmetric spectra and their properties.

The correct notion of homotopy for spectra is the stable homotopy group defined to be
\[\pi_n^s({\bf X}) = \lim_{k} \pi_{n+k} (X_k)\]
where the limit is taken over the structure maps $\sigma$.
A map $f:{\bf X} \to {\bf Y}$ of symmetric spectra is given by a collection of $\Sigma_n$-equivariant maps $f_n:X_n \to Y_n$ which {\it strictly} commute with the structure maps $\sigma$.  If we have a commuting ladder of long exact sequences
\[ \xymatrix{ \cdots \ar[r]& \pi_k(X_n) \ar[r]^{f_*} \ar[d]^{\sigma} & \pi_k(Y_n)\ar[r] \ar[d]^{\sigma} & \pi_k(Y_n, X_n)\ar[r]\ar[d] 
%&\pi_{k-1}(A_m)\ar[r] \ar[d] 
& \dots \\
\dots \ar[r] & \pi_{k+1}(X_{n+1}) \ar[r]^{f_*} & \pi_{k+1}(Y_{n+1}) \ar[r] & \pi_{k+1}(Y_{n+1}, X_{n+1})\ar[r] 
%& \pi_k(E_{m+1})\ar[r] 
& \dots}\]
for all $k$ and $n$, then we obtain a long exact sequence of stable homotopy groups associated to $f$ with 
\[ \pi_k^s({\bf Y}, {\bf X}) = \lim_n \pi_{k+n} ( Y_n, X_n).\]

As before, we suppose that we have a tower 
\[ \xymatrix{ \cdots \ar[r] &{\bf W_{m+1}}\ar[r]^{q_{m+1}} & {\bf W_m} \ar[r]^{q_m\quad} & {\bf W_{m-1}} \ar[r] & \cdots}\]
where each ${\bf W_i}$ is now a symmetric spectrum rather than a pointed space.    The homotopy fiber of $f_n:X_n \to Y_n$ is  the (strict) pullback:
\[ \xymatrix{ F_n\ar@{-->}[r] \ar@{-->}[d]_{d_n} & PY_n\ar[d]^{ev_1} \\
X_n \ar[r]^{f_n} & Y_n}\]
where $PY_n$ is the based path space of $Y_n$ and $ev_1$ is the evaluation of the path $\gamma:I\to Y_n$ at $t=1$ (we take $t=0$ to be the basepoint of $I$).
The pullback is taken in the category of $\Sigma_n$-equivariant basepointed spaces, so that $F_n$ naturally obtains a basepoint-preserving action of $\Sigma_n$.
Taking the homotopy fiber of each $f_n$ produces a spectrum ${\bf F}=\{ F_n\}$ whose structure maps $\sigma_F$ are constructed via the universal property of the pullback.  Let $eval:\Sigma PY_n\to P(\Sigma Y_n)$ be the map taking $s\wedge \gamma\in S^1\wedge PY_n$ to the path $s\wedge \gamma$ in $\Sigma Y_n$.  The the outer square of the diagram
\[ \xymatrix{
\Sigma F_n \ar[r] \ar@{-->}[dr]^{\sigma_F} \ar[dd]_{\Sigma d_n} &\Sigma PY_n \ar[r]^{eval} & P(\Sigma Y_n) \ar[d]^{P(\sigma_Y)} \\
 & F_{n+1} \ar[r] \ar[d]_{d_{n+1}} & PY_{n+1} \ar[d]^{ev_1}\\
 \Sigma X_n \ar[r]^{\sigma_X} &X_{n+1} \ar[r]^{f_{n+1}} & Y_{n+1} }\]
is seen to commute after a short diagram chase, ensuring the existence of $\sigma_F$.  The $\Sigma_n$-equivariance of the maps $\sigma_X$ and $d_n$, plus the $\Sigma_{n+1}$-equivariance of $d_{n+1}$ ensures that $\sigma_F$ is $\Sigma_n$-equivariant (or $\Sigma_1\times \Sigma_n$-equivariant, if you prefer).  Similarly, the maps $\sigma^r_F:\Sigma^rF_n\to F_{n+r}$ are seen to be $\Sigma_r\times \Sigma_n$-equivariant because of the equivariance of the maps $\Sigma^rd_n$, $\sigma^r_X$ and $d_{n+r}$ in the commuting diagram
\[ \xymatrix{
\Sigma^rF_n \ar[r]^{\sigma^r_F} \ar[d]_{\Sigma^r d_n} & F_{n+r} \ar[d]^{d_{n+r}} \\ \Sigma^rX_n \ar[r]^{\sigma^r_X} & X_{n+r}}.\]   We obtain isomorphisms $\pi_k^s({\bf Y}, {\bf X}) \cong \pi_k ({\bf F})$ in the same way we did for spaces.

Finally, we require the operad action of ${\bf P}$ on the tower ${\bf W_*}$ to be strict; that is, the structure maps should produce strictly commuting diagrams.

\begin{theorem}\label{t:spectra} If ${\bf P}$ acts on a tower of fibrations ${\bf W_*}$ of spectra, then the associated spectral sequence is  a sequence of $\pi_*({\bf P})$-algebras.
\end{theorem}

Under these assumptions, the proof of this theorem is exactly as  given for spaces in Section \ref{s:spaces}.   We note, in particular, that this theorem implies that multiplicative towers of spectra with strictly commuting multiplicative structure whose multiplication is only ( for e.g) homotopy associative produce homotopy associative multiplicative spectral sequences,  since an operad will classify the associative multiplication up to homotopy.

We also note that the convergence statements of subsection 3.2 and remark \ref{r:convergence}  have analogues in the stable setting.  We leave the formulation and proofs of these statements as an exercise for the reader.

\subsection{Examples}  Given a spectrum ${\bf E}$, the generalized homology of the spectrum ${\bf X}$ with coefficients in ${\bf E}$ is defined by 
\[ {\bf E}_n({\bf X}) = \pi_n({\bf E\wedge X}).\]  When ${\bf E}$ is the Eilenberg-MacLane spectrum ${\bf H{\mathbb Z}/p}$, the resulting homology theory is the mod $p$ homology of the spectrum ${\bf X}$, which is a homology theory of much study (see, e.g. \cite {A} or \cite{R}).
If ${\bf E}$ is also a commutative and associative ring spectrum, then an operad action on ${\bf X}$ passes to an operad action on ${\bf E\wedge X}$ via the maps
\[ \xymatrix{ \bf{P(k)} \wedge ({\bf E\wedge X})  \wedge \cdots \wedge ({\bf E\wedge X}) \ar[r] & 
{\bf E}^{\wedge k} \wedge ({\bf P(k)} \wedge  {\bf X}^{\wedge k}) \ar [r]^{\quad\quad \mu \wedge \rho}&
{\bf E\wedge X}} \] 
where the map $\mu$ is the multiplication map for the ring spectrum ${\bf E}$ and $\rho$ is the structure map for the operad algebra multiplication for ${\bf X}$.

 A particularly interesting case is given by the Goodwillie tower of fibrations associated to the functor $\Sigma^{\infty}Maps(K, X)_+$ as a functor of $X$ and for a fixed finite complex $K$.  The tower of fibrations
 \[ \xymatrix{ && \cdots \ar[d] \\ 
 && P_2^K(X) \ar[d] \\
 && P_1^K(X) \ar[d] \\
 \Sigma^{\infty}Maps(K,X)_+ \ar[rr] \ar[urr]\ar[uurr] && P_0^K(X) } \]
 was computed by G. Arone \cite{Ar}.  Later, S. Ahearn and N. Kuhn \cite{AK} studied the structure of spectral sequences associated to this tower for various values of $K$.  In particular, they showed that the little cubes operad ${\mathcal C}_n$ acts on the tower $P_*^{S^n}(X)$, where  ${\mathcal C}_n$ 
 is the operad whose $k$-th space ${\mathcal C}_n(k)$ consists of an ordered collection of $k$ little $n$-cubes linearly embedded into the standard $n$-cube with disjoint interiors and axes parallel to those of the standard $n$-cube;  this operad is the standard one that recognizes an $ n$-fold loop space, \cite{M, MSS}.     %The authors note that ``Computationally, this implies that the associated spectral sequences for computing mod $p$ homology admit Dyer-Lashoff operations".  However, in a footnote they acknowledge that ``Exactly what this means is still a matter of investigation by the authors".

Now if 
 W is the standard  ${\Bbb Z}[{\Bbb Z}_p]$-free resolution, there is a ${\Bbb Z}_p$-equivariant map $W^{(n)}_* \to C_* {\mathcal C}_{n+1}(p)$.   Then on the singular chains $K = C_*(P_*^{S^n}(X))$ of the tower, we have an induced operad action of $W$,  given by the ${\Bbb Z}_p$-equivariant map $\theta:  W^{(n)}_*  \otimes K^{\otimes p}  \to K$  coming from the action of the little cubes operad.  This operad action induces  the Dyer-Lashof operations  for $p$ odd prime\cite{M}. 
Applying generalized mod-p homology ${\bf H_*} = \pi_*({\bf H{\mathbb Z}/p}\wedge - )$
and 
Theorem \ref{t:spectra}, we obtain a spectral sequence of algebras over 
${\bf H}_*({\mathcal C}_n)$.  Since the Dyer-Lashof operations are induced by the action of the little cubes operad, this explains the authors' note that ``Computationally, this implies that the associated spectral sequences for computing mod $p$ homology admit Dyer-Lashof operations." \cite{AK}.

%---------------------------------

%------------------------------

\end{document}